\documentclass[11pt]{amsart} 

\usepackage{amsmath}
\usepackage{bm}
\usepackage{tikz}
\usepackage{pgfplots}

\usepackage{verbatim}
\usepackage{graphicx}
\usepackage{amsmath, amssymb}
\usepackage{multirow}
\usepackage{accents}
\usepackage{hhline}

\usepackage{amsmath,mathtools}
\usepackage{amsfonts}
\usepackage{amssymb}
\usepackage{subfigure}

\usepackage[colorlinks,bookmarksopen,bookmarksnumbered,citecolor=red,urlcolor=red]{hyperref}

\newtheorem{theorem}{Theorem}[section]

\newtheorem{lemma}[theorem]{Lemma}

\def\A{{\mathcal A}}
\def\R{{\mathbb R}}
\def\S{{\mathcal S}}

\newlength{\figurewidth}
\newlength{\figurewidththree}
\begin{document}

\title[Numerical study elliptic equation with fractional BC]
{A numerical study of the homogeneous elliptic equation with fractional order boundary conditions}

\author{Raytcho Lazarov \and Petr Vabishchevich}
\address{R. Lazarov, Department of Mathematics, Texas A\&M University, College Station, TX 778 \and
Institute of Mathematics and Informatics, Bulgarian Academy of Sciences, ul.~Acad.~G.Bonchev, bl. 8, Sofia, Bulgaria 
{(lazarov@math.tamu.edu)}}
\address{P. Vabishchevich, Nuclear Safety Institute of RAS, 52, B. Tulskaya, Moscow, Russia
\and Peoples' Friendship University of Russia (PRUDN University)\\
6, Miklukho-Maklaya Str., Moscow, RUSSIA
{(vabishchevich@gmail.com)}}

 \maketitle
 


 \begin{abstract}
We consider the homogeneous equation $\A u=0$, where $\A$ is a symmetric and coercive elliptic operator
in $H^1(\Omega)$ with $\Omega$ bounded domain in ${{\mathbb R}}^d$.
The boundary conditions involve fractional power $\alpha$, $ 0 < \alpha <1$, of the
Steklov spectral operator arising in Dirichlet to Neumann map.
For such problems we discuss two different numerical methods:
(1) a computational algorithm based on an approximation of the integral representation 
of the fractional power of the operator and
(2) numerical technique involving an auxiliary Cauchy problem for an ultra-parabolic
equation and its subsequent approximation by a time stepping technique.
For both methods we present numerical experiment for a model two-dimensional problem that 
demonstrate the accuracy, efficiency, and stability of the algorithms.
\end{abstract}

\medskip

 \smallskip




 \section{Introduction}\label{sec:1}

In the last two decades a number of nonlocal differential operators generated by 
fractional derivatives in time and space have been used to model various applied
problems in applied physics, biology, geology, finance, and engineering, 
e.g. \cite{metzler2014anomalous,podlubny1998fractional}.
Many models 
involve both sub-diffusion (fractional in time) and super-diffusion (fractional in space) differential operators. 
Often, super-diffusion problems are treated as problems with a fractional power of an elliptic operator.
Loosely speaking for $\A: H^1_0(\Omega) \to H^{-1}(\Omega)$ determined by the identity 
$(\A  u, v) = ( \nabla u, \nabla v) $ for $u,v \in H^1_0(\Omega)$ 
its fractional power is defined by
\begin{equation}\label{eq:f-oper}
\A^\alpha u= \sum_{j=1}^\infty \lambda^\alpha_j (u, \psi_j) \psi_j, \quad \mbox{where} \quad 
\A \psi_j =\lambda_j \psi_j
\end{equation}
i.e. $\lambda_j, \psi_j$ are the eigenvalues and normalized eigenfunctions of the 
negative Laplacian with homogeneous Dirichlet boundary conditions.

To discretize the problem $\A^\alpha u= f $ 
we can apply finite volume or finite element methods to get an approximation $\A_h$ of  $\A$
and subsequently to reduce the problem to a discrete form $\A^\alpha_h u_h=f_h$.
A practical implementation of such approach requires the matrix function-vector multiplication
$\A^\alpha_h u_h$ where the matrix $\A_h$ is not known explicitly.
For such problems, different approaches \cite{higham2008functions} are available.
Algorithms for solving systems of linear equations associated with
fractional elliptic equations that are 
based on Krylov subspace methods with the Lanczos approximation
are discussed, e.g., in \cite{ilic2009numerical}.
A comparative analysis of the contour integral method, the extended Krylov subspace method, and the preassigned 
poles and interpolation nodes method for solving
space-fractional reaction-diffusion equations is presented in \cite{burrage2012efficient}.
The simplest variant is associated with the explicit construction of the solution using the 
eigenvalues and eigenfunctions of the elliptic operator with diagonalization of the corresponding matrix
\cite{bueno2012fourier,ilic2006numerical}. 
Unfortunately, all these approaches demonstrates quite high computational complexity for multidimensional problems.
In the special case when there is an efficient method for solving the equation $\A_h w_h=f_h$, an algorithm based 
the best ratioinal approximation of $t^{-\alpha}$ on $(0,1)$ has been proposed and 
experimentally justified in \cite{Harisanov2017}.

One can adopt a general approach to solve numerically equations involving fractional power of operators
by first approximating the original operator and then taking fractional power of its discrete variant. 
Using Dunford-Cauchy formula the elliptic operator is represented as a contour integral  in the complex plane. 
Further applying appropriate quadratures with integration nodes in the complex plane one ends up with 
a proper method that involves only inversion of the original operator. 
The approximate operator is treated as a sum of resolvents \cite{gavrilyuk2004data,gavrilyuk2005data}
ensuring the exponential convergence of quadrature approximations.
Bonito and Pasciak in \cite{bonito2015numerical} presented a more promising variant of using quadrature
formulas with nodes on the real axis, which are constructed
on the basis of the corresponding integral representation for the power operator  \cite{krasnoselskii1976integral}.
In this case, the inverse operator of the problem has an additive representation, where
each term is an  inverse of the original elliptic operator.
A similar rational approximation to the fractional Laplacian operator is studied in \cite{AcetoNovat}.

In \cite{vabishchevich2014numerical} a computational algorithm for solving
an equation with fractional powers of elliptic operators on the basis of
a transition to a pseudo-parabolic equation has been proposed, see equation \eqref{19}.
For the auxiliary Cauchy problem, standard two-level schemes are applied.
The computational algorithm is simple for practical use, robust, and applicable to solving
a wide class of problems.  One needs to theoretical study the stability and the convergence 
of such schemes. The case of smooth data could be studied with the existing 
methods, see, e.g. \cite{Samarskii1989, Thomee2006}, while the case of non-smooth
data needs deeper and more refined analysis. 
The computations  in \cite{vabishchevich2014numerical} show that
usually, a small number of pseudo-time steps is required to get  
a good approximation of the required solution of the discrete fractional equation.
This computational algorithm for solving equations with fractional powers of operators
is promising also when considering transient problems. 

We note that \eqref{eq:f-oper} gives one possible definition of the fractional power
of elliptic differential operators. Another possibility is to define it via Ritz potentials,  for 
$ \Omega ={{\mathbb R}}^d$, 
see, e.g. \cite[Section 2.10, p.~128]{kilbas2006theory}.  The extension of such derivative to bounded domain $\Omega$
has been used in the work of \cite{bonito2015numerical}. 
This quite general definition could be used also for complex values of $\alpha$. 

In the present study, we consider a new problem, where the solution inside a 
domain satisfies the homogeneous elliptic equation of second order with a given 
fractional boundary condition, introduced in the following manner. 
The Dirichlet to Neumann map evaluates the normal derivative of a harmonic
function for given Dirichlet data and defines 
an operator $\S: H^{\frac12}(\Gamma) \to H^{-\frac12}(\Gamma)$ 
on the dense set, for example, $ D(\mathcal{S} ) = \{ u:  \ u( x) \in H^1(\Gamma)\} \subset L^2(\Gamma)$.
The fractional power $0< \alpha <1$ of the operator $\S$ is defined through the 
eigenvalues and the eigenfunctions of the corresponding Steklov spectral problem, 
see, e.g.  \cite{babuska1991eigenvalue-art}. This idea 
is explained in details in Section \ref{sec:problem}.

The main contribution of this paper is construction and testing of two numerical algorithms 
for computing efficiently an approximation of the equation $\S^\alpha u(x) =g(x)$, $x \in \Omega$, for $0< \alpha < 1$.
This is novel class of mathematical problems where the fractional boundary condition  
is formulated on the basis of the Steklov spectral problem. The standard mathematical problems
of finding the solution of homogeneous elliptic equation   
with Dirichlet or Neumann-type boundary conditions are two 
limiting cases, $\alpha = 0$  and $\alpha = 1$, correspondingly. 
 
To solve approximately problems with fractional boundary conditions, we use the standard
space of piece-wise polynomial functions on a quasi-uniform   partition 
of the domain into simplexes, see, e.g. \cite{Thomee2006}. Further, we develop and test 
 two computational algorithms, one based on an approximation of 
  the corresponding operator of fractional power and  second one,  based on the solution of 
  the auxiliary Cauchy problem for the pseudo-parabolic equation. 
 Finally, we present a number of numerical experiments on some model two-dimensional problem
 that demonstrate the efficiency  and the accuracy of the methods on smooth data. The paper is rather a proof
 of a {\it computational concept} than rigorous study of accuracy and the convergence 
 of the proposed methods. We are confident  that the proposed computational approach merits
 rigorous error analysis, especially for non-smooth solutions, and 
 possible extension to more general elliptic operators.

To reduce the complexity of the notation in the 
paper we use the a calligraphic letters for denoting operators in infinite dimensional spaces and usual 
capital letters for their finite dimensional approximations, e.g. $\S$ denotes the Dirichlet to Neumann
map, while $S$ denotes its finite element approximation.

\section{Problem formulation} \label{sec:problem}

In a bounded domain $\Omega \subset \R^d$, $d=1,2,3$ with the Lipschitz continuous boundary 
$\Gamma \equiv \partial\Omega$, we consider the following operator defined by:
\begin{equation}\label{eq:oper}
(\A  u, v) = a(u,v) := \int_\Omega \left ( k(x) \nabla u \cdot \nabla v + c(x) uv \right ) dx  
\quad \forall u,v \in H^1(\Omega),
\end{equation}
where $k(x) \ge k_0 >0$ and $c(x) \ge 0$ for $x \in \Omega$. We further assume that
$\A$ is coercive in $H^1(\Omega)$ so that $(\A u,u) \ge \delta (u,u)$.
 Then for a given suitably smooth data $ g(x)$,  $ x \in \Gamma$, 
 the problem find $u \in H^1(\Omega)$ such that 
 \begin{equation}\label{eq:weak}
 a(u,v) = \langle g, v \rangle_\Gamma := \int_\Gamma g v ds, \quad \forall v \in H^1(\Omega)
  \end{equation}
has unique solution  $u \in H^1(\Omega)$.  The trace of 
$v$ on $\Gamma$ belongs to the Sobolev space $H^\frac12(\Gamma)$
so  we can defined the operator  $\S: H^{\frac12}(\Gamma) \to H^{-\frac12}(\Gamma)$ 
by the identity $a(u,v)=\langle \S u, v \rangle_\Gamma$ for all $v \in H^1(\Omega)$.
Then on $\Gamma$ the solution of equation \eqref{eq:weak} satisfies 
$\S u=g$. Here the operator $\S$ is the well-known Dirichlet to Neumann map.

Now we introduce the problem we intend to study, namely,  for a given suitably smooth data $g$ 
we seek the solution of the operator equation 
\begin{equation}\label{eq:frac}
\S^\alpha u=g, \quad \mbox{ where } \quad 0 < \alpha <1.
\end{equation}
To define
the fractional power $\S^\alpha$ we first introduce the Steklov type eigenvalue problem,
e.g.  \cite{babuska1991eigenvalue-art}: find $\psi_j \in H^1(\Omega)$
and $\lambda_j \in \R$ so that 
\begin{equation}\label{eq:spec}
a(\psi_j,v) = \lambda_j \langle \psi_j, v \rangle_\Gamma \quad \forall  v \in H^1(\Omega).
\end{equation}
It is well known, e.g. \cite{babuska1991eigenvalue-art}, 
 that this spectral problem has full set of eigenfunctions that span the space $H^1(\Omega)$ so
that we can define the fractional powers of $\S$ in the same manner as for general symmetric elliptic operators,
namely,
$$
\S^\alpha u = \sum_{j=1}^\infty \lambda^\alpha_j  \langle \psi_j, u \rangle_\Gamma \, \psi_j.
$$
The operator $\S$,  defined on the domain
\[
 D(\S ) = \{ u:  \ u( x) \in L_2(\Gamma), \ \sum_{j=0}^{\infty} \lambda_j   \langle u, \psi_j \rangle_\Gamma^2 < \infty \} ,
\]
is self-adjoint and coercive in $L^2(\Gamma)$
\begin{equation}\label{SS6}
  \mathcal{S}  = \mathcal{S} ^* \geq \delta I.
  \quad \delta > 0 ,    
\end{equation} 
Here $I$ is the identity operator in $L^2(\Gamma)$.
For $\delta$, we have $\delta = \lambda_1$.
In applications, the value of $\lambda_1$ is unknown.  However, one can 
find a reliable positive bound from below.

\section{Finite element approximation} 

We consider a standard quasi-uniform triangulation of the domain $\Omega$ 
into triangles (or tetrahedra in 3-D). Let $ x_i, \ i = 1,2, ..., N_h$ be 
vertexes of this triangulation. We introduce the finite dimensional space $V_h \subset H^1(\Omega)$
of continuous functions that are liner over each finite element, see, e.g. \cite{Thomee2006}. 
As a nodal basis we take the standard ``hat" function $\chi_i(x) \in V_h, \ i = 1,2, ..., N_h$.
Then for $v \in V_h$, we have the representation
\[
 v(x) = \sum_{i=i}^{N_h} v_i \chi_i( x) := \sum_{i=i}^{N_h} v(x_i) \chi_i( x).
\] 

Then the corresponding approximations of equation (\ref{eq:weak}) is: find $y \in V_h$ such that 
\begin{equation}\label{8}
 \langle S y, v \rangle = \langle g, v \rangle \quad \mbox{where} \quad  \langle S y,v \rangle_\Gamma = a(y,v) 
 \quad \forall \ v \in V_h. 
 \end{equation} 
 Here $g$ is the given boundary data, see \eqref{eq:weak}. 
Similarly, the approximation of the spectral problem \eqref{eq:spec} is 
 \[
 S \widetilde{\psi}_j = \widetilde{\lambda}_j \widetilde{\psi}_j.
\] 
The eigenpairs $( \widetilde{\lambda}_j , \widetilde{\psi}_j)$, $j=1, \dots, N_\Gamma$, 
have the following properties (see, e.g., \cite{armentano2004effect})
\[
\widetilde{\lambda}_1 \leq \widetilde{\lambda}_2 \leq ... \leq  \widetilde{\lambda}_{N_\Gamma},
\quad \| \widetilde{\psi}_j\| = 1,
\quad j = 1,2, ..., N_\Gamma , 
\]
with $N_\Gamma$  being the number of vertexes on the boundary $\Gamma$.

The operator $S$ acts on a finite dimensional sub-space of $V_h$ 
and, similarly to inequality (\ref{SS6}), we have
\begin{equation}\label{9}
S = S^* \geq \delta I ,
\quad \delta > 0 , 
\end{equation} 
where $\delta \leq \lambda_1 \leq \widetilde{\lambda}_1$. 
The fractional power of the operator $S$ is defined by
\[
 S^{\alpha } y = \sum_{j=1}^{N_h}  \langle y, \widetilde{\psi}_j \rangle_\Gamma  \widetilde{\lambda}_j^{\alpha } 
 \ \widetilde{\psi}_j 
\] 
and  the corresponding finite element approximation of equation (\ref{eq:frac})  is
\begin{equation}\label{10}
 S^\alpha y = g .
\end{equation} 
In fact, since 
 the identity \eqref{8} is over $v \in V_h$, instead of $g$ here we should have the orthogonal $L^2$-projection
 of $g$ onto the trace of $V_h$ on $\Gamma$.  Using the same letter for the 
 original data and for its projection on the finite element space leads to some ambiguity,
 but it simplifies the notations and we hope it does not lead to confusion.
 In  view of (\ref{9}), for the solution (\ref{10}) we get the following trivial a priori estimate:
\begin{equation}\label{11}
 \|y\|_\Gamma \leq \delta^{-\alpha} \|g\|_\Gamma.
\end{equation} 

\section{ Method I. Approximation of the fractional power of a symmetric positive operator using integral representation}
%
Here we construct a numerical algorithm for solving  \eqref{9} that uses
an approximation for $S^\alpha$ using its integral representation
(see, e.g., \cite{krasnoselskii1976integral}):
\begin{equation}\label{13}
 S^{-\alpha} = \frac{\sin(\pi \alpha)}{\pi} \int_{0}^{\infty} \theta^{-\alpha} (S + \theta  I)^{-1} d \theta ,
 \quad 0 < \alpha < 1 . 
\end{equation} 
The approximation of $S^{-\alpha}$ is based on the use of one or another quadrature formulas for
the right-hand side of (\ref{13}). Various possibilities in this 
direction are discussed in \cite{bonito2015numerical}.
One possibility is the special quadrature Gauss-Jacobi formulas used in \cite{AcetoNovat}. 
Here in this algorithm  we apply an exponentially convergent quadrature formula 
considered and studied in \cite{bonito2015numerical}.

In  (\ref{13}) we introduce a new variable $s$, 
$\theta = e^{-2 s}$, so that 
\begin{equation}\label{14}
 S^{-\alpha} = 2 \frac{\sin(\pi \alpha)}{\pi} \int_{-\infty}^{\infty} e^{2\alpha s} (I + e^{2 s} S)^{-1} d s ,
 \quad 0 < \alpha < 1 . 
\end{equation}
Obviously, the main task here is to select a good approximation and fast evaluation the right-hand side of (\ref{14}).

Following \cite{bonito2015numerical},
we apply 
the quadrature formula of rectangles
with nodes $s_m = m \eta, \ m = - M, -M+1, ..., M$ for $\eta = M^{-1/2}$  to  get
the following approximation of equation \eqref{10}
\begin{equation}\label{12}
 D_M y_M = 
 g,   \qquad  y_K \approx  y,   \qquad  S^{-\alpha} \approx D_{M}^{-1} ,
\end{equation}
where 
\begin{equation*}\label{16}
 D_{M}^{-1} = 2 \eta \frac{\sin(\pi \alpha)}{\pi} \sum_{m = -M}^{M} e^{2\alpha s_m} (I + e^{2 s_m} S)^{-1} .
\end{equation*} 
This could be rewritten as 
\begin{equation}\label{17}
 D_{M}^{-1} = \sum_{m = -M}^{M} S_{m}^{-1},
 \quad  S_{m} = S_{m}^{*} > 0 ,
\end{equation}
where for the individual terms of the operator, we have 
$ 
 S_{m} = a_m S + b_m I, $
 $ a_m  > 0,  $
 $ b_m > 0, $
 $  m = - M, -M+1, ..., M .
$ 
In view of this, from (\ref{12}), it follows that
\begin{equation}\label{18}
 y_M = \sum_{m = -M}^{M} S_{m}^{-1} g . 
\end{equation}
The approximate solution is determined as the solution of $2M + 1$ standard problems
with operators $S + c_m I, \ c_m > 0, \ m = - M, -M+1, ..., M$.

The numerical method involves
solving a number of elliptic problems with
Neumann boundary conditions. Indeed, from (\ref{18}), we have 
\[
 y_M = \sum_{m = -M}^{M} y_m,  
 \quad \mbox{where} \quad
 S_{m} y_m = g.
\] 
In view of the above notation, we have
$
 a_m  \langle S y_m, v \rangle_\Gamma + b_m  \langle y_m, v \rangle_\Gamma =  \langle g, v \rangle_\Gamma,
 $ i.e. for each $m$ we solve the following standard system: find $y_m \in V_h$ s.t. 
\[
 a_m a(y_m, v) + b_m  \langle y_m, v \rangle_\Gamma =  \langle g, v\rangle_\Gamma, \ \ \forall v \in V_h.
 \]

\section{ Method II. Approximation of the fractional power of a 
symmetric positive operator using pseudo-parabolic problem} 

Now we present a second algorithm for solving approximately problem \eqref{10}
based on its equivalence to to solution of an auxiliary pseudo-time  
evolutionary 
problem \cite{vabishchevich2014numerical}.
Let $w(x,t)$ be a function defined on $\Gamma \times [0,1]$ such that $w(t) \in L^2(\Gamma)$ for 
any $t \in [0,1]$ and
\[
 w(t) = \delta^{\alpha } (t D + \delta I)^{-\alpha} w(0)  \quad \mbox{with} \quad  D = S - \delta I.
\]
Due to \eqref{9}, we have 
$ 
 D = D^* > 0 .
$ 
By this construction
$ 
 w(1) =  \delta^{\alpha} S^{-\alpha} w(0)
$ 
and comparing it with the solution of equation \eqref{10} we see that  
if we take $w(0)= \delta^{-\alpha} g$ then
$w(1) = S^{-\alpha} g = y$, i.e. this is the solution of \eqref{10}. It is also easy to see that 
 $w(t)$ satisfies the following  pseudo-parabolic initial value problem
\begin{equation}\label{19}
  (t D + \delta I) \frac{d w}{d t} + \alpha D w = 0 ,
  \quad 0 < t \leq 1, \quad  w(0) = \delta^{-\alpha} g.
  \end{equation}  
Therefore, the solution of equation (\ref{10}) coincides with the solution of the Cauchy problem 
(\ref{19}) at pseudo-time moment $t=1$.

We can obtain various a priori estimates for 
(\ref{19}). The estimate  similar to (\ref{11}) has the form
\begin{equation}\label{22}
  \|w(t)\|  \leq \|w(0)\|  = \delta^{-1} \|g\|, \quad 
  \mbox{where} \quad \|w(t)\| =\left (\int_\Gamma |w(t)|^2ds
  \right)^\frac12 .
\end{equation}

To solve numerically the problem (\ref{19}),
we apply implicit two-level scheme, see, e.g. \cite{Samarskii1989}.
Let $\tau$ be the step-size of a uniform grid in time such that 
$w^n = w(t^n), \ t^n = n \tau$, $n = 0,1, ..., N, \ N\tau = 1$.
We approximate equation (\ref{19}) by the following implicit two-level scheme
\begin{equation}\label{23}
 (t_\sigma^n D + \delta I) \frac{ w^{n+1} - w^{n}}{\tau }
 + \alpha D w_\sigma^n = 0,  \quad n = 0,1, ..., N-1,
\end{equation}
\begin{equation}\label{24}
 w^0 = \delta^{-\alpha} g ,
\end{equation} 
where $ 0 < \sigma \le 1$ is a parameter  and
\[
  t^n_{\sigma}= \sigma t^{n+1} + (1-\sigma) t^{n},
  \quad w_\sigma^{n} = \sigma w^{n+1} + (1-\sigma) w^{n}.
\]

The stability of the scheme is established in the  following Lemma:
\begin{lemma} \label{lem:stab}
For $\sigma \geq 0.5$, the difference scheme (\ref{23}), (\ref{24}) 
is unconditionally stable with respect to the initial data and the approximate solution satisfies the 
a priori estimate 
\begin{equation}\label{25}
  \|w^{n+1}\|_\Gamma  \leq \|w^0\|_\Gamma  , 
  \quad n = 0,1, ..., N-1.
\end{equation} 
\end{lemma}
\begin{proof}
To prove this statement we first rewrite equation (\ref{23}) in the form:
\[
 \delta \frac{ w^{n+1} - w^{n}}{\tau } + D \left( \alpha w_\sigma^n 
      + t_\sigma^n\frac{ w^{n+1} - w^{n}}{\tau } \right ) = 0 . 
\] 
Multiplying this it by
$
 \alpha w_\sigma^n + t_\sigma^n  \left( w^{n+1} - w^{n} \right)/\tau,
$
taking a discrete $L^2(\Gamma)$-inner product in view of the positivity of $D$ we get 
\[
 \left ( \frac{ w^{n+1} - w^{n}}{\tau }, w_\sigma^n \right )_\Gamma  \leq 0. 
\] 
Further, since 
$\displaystyle
 w_\sigma^n = \tau \Big ( \sigma - \frac{1}{2} \Big ) \frac{w^{n+1} - w^{n}}{\tau }   +
 \frac{1}{2} (w^{n+1} + w^{n}),
$ 
then for  $\sigma \geq 0.5$ we easily get \eqref{25}.
\end{proof} 

Lemma \ref{lem:stab}  and the approximation properties of the 
finite difference scheme (\ref{23}), (\ref{24}) ensures that for sufficiently smooth $w(t)$
its approximate solution converges to $w(t)$ with second order 
for $\sigma=0.5$ 
and with first  order for all other  values of $\sigma$.
The smoothness of $w(t)$ depends on the smoothness of the data $g$ and the 
properties of the pseudo-parabolic problem.
The case of non-smooth solutions (or non-smooth data) is a subject of a separate  study.

\section{Numerical experiments} 

Here we present results of the numerical solution of a model problem in two spatial dimensions,
where the computational domain is a quarter of the circle of radius 1.
We solve the problem for the elliptic problem (\ref{eq:weak}) for
$ 
 k({x}) = 1,  c({ x}) = c_0 = const > 0$,
$  x \in \Omega$,  and $ g( x) = 1$ for $ x \in \Gamma . $ 
In the numerical experiments for testing the algorithm based on solving the Cauchy problem
we choose $\sigma =0.5$. For smooth solutions the scheme has second order approximation in time.

\begin{figure}[h!]
\begin{minipage}{0.4\linewidth}
  \begin{center}
    \includegraphics[width=1.\linewidth] {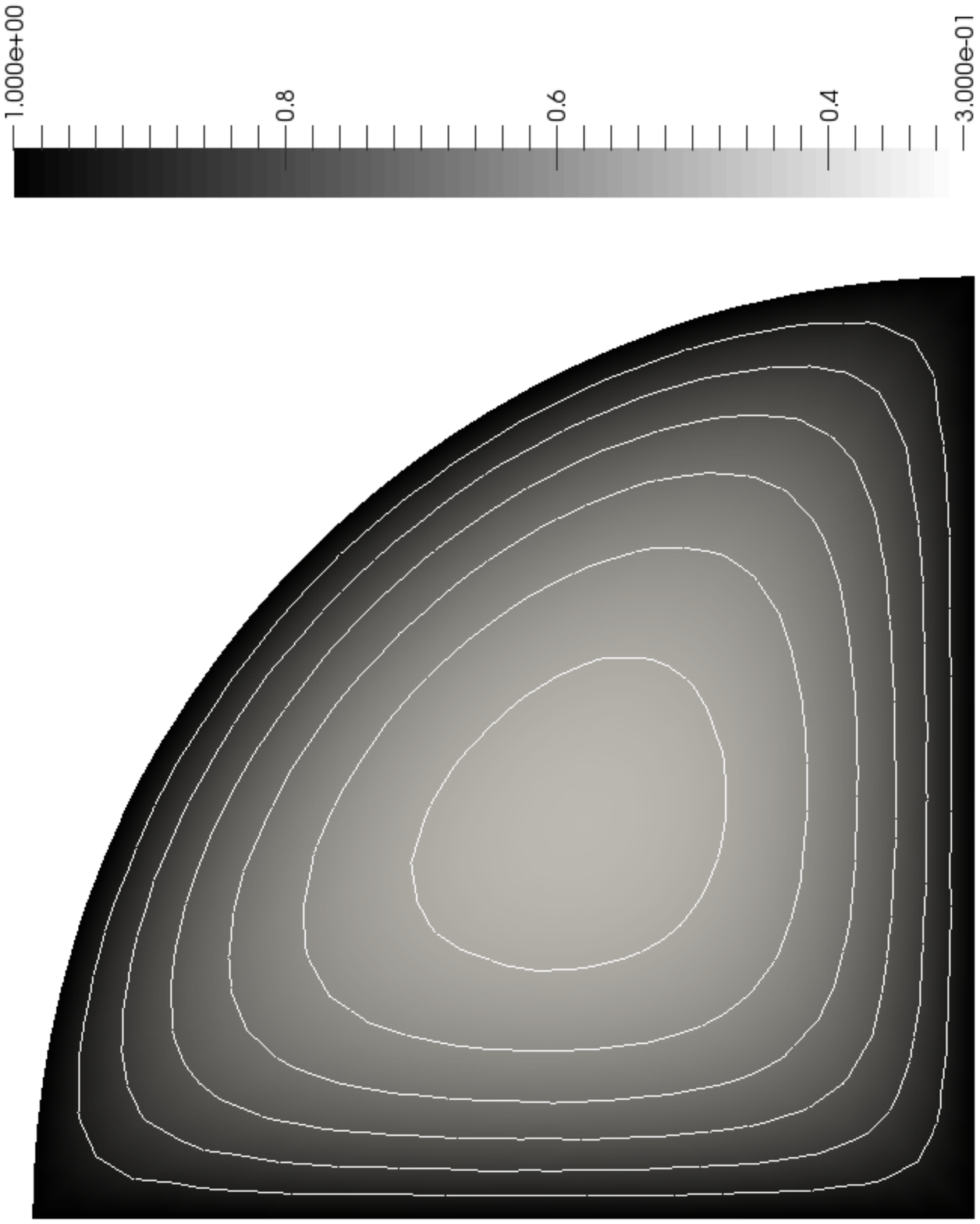}
  \end{center}
\end{minipage}\hfill
\begin{minipage}{0.4\linewidth}
  \begin{center}
    \includegraphics[width=1.\linewidth] {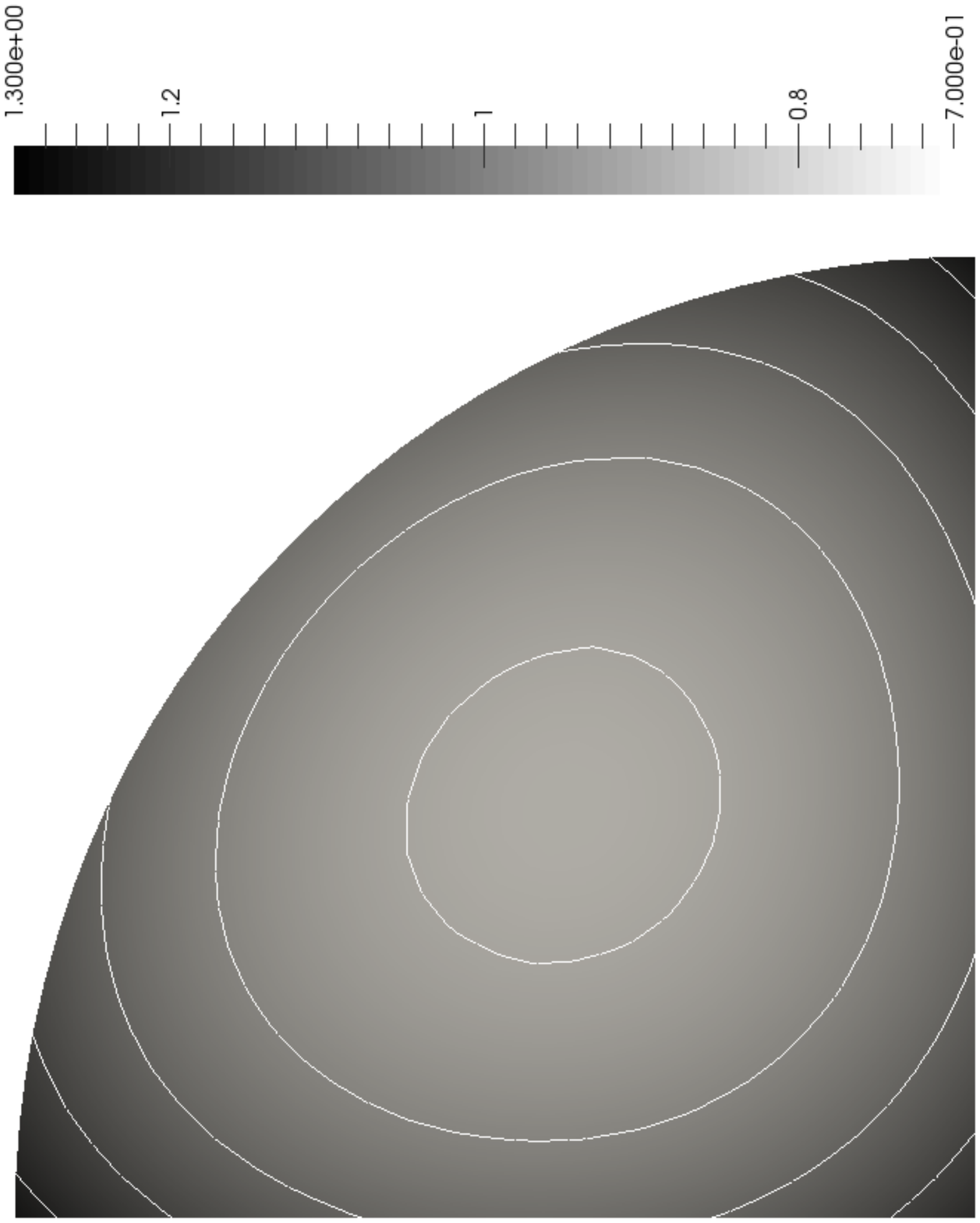}
  \end{center}
\end{minipage}
	\caption{Approximate solutions for $c_0=5$ in 
	two limiting cases: (left) $\alpha = 0$ (Dirichlet data), $\min w( x) = 0.38688$, 
	 $\max w( x) = 1$ (achieved on $\Gamma$)
         and (right) $\alpha = 1$ (Neumann data), $\min w( x) = 0.7741$,  $\max w( x) = 1.2417$.}
	\label{fig.1}
\end{figure}
\begin{figure}[h!]
\begin{minipage}{0.33\linewidth}
  \begin{center}
    \hspace{20mm}
    \includegraphics[width=.9\linewidth] {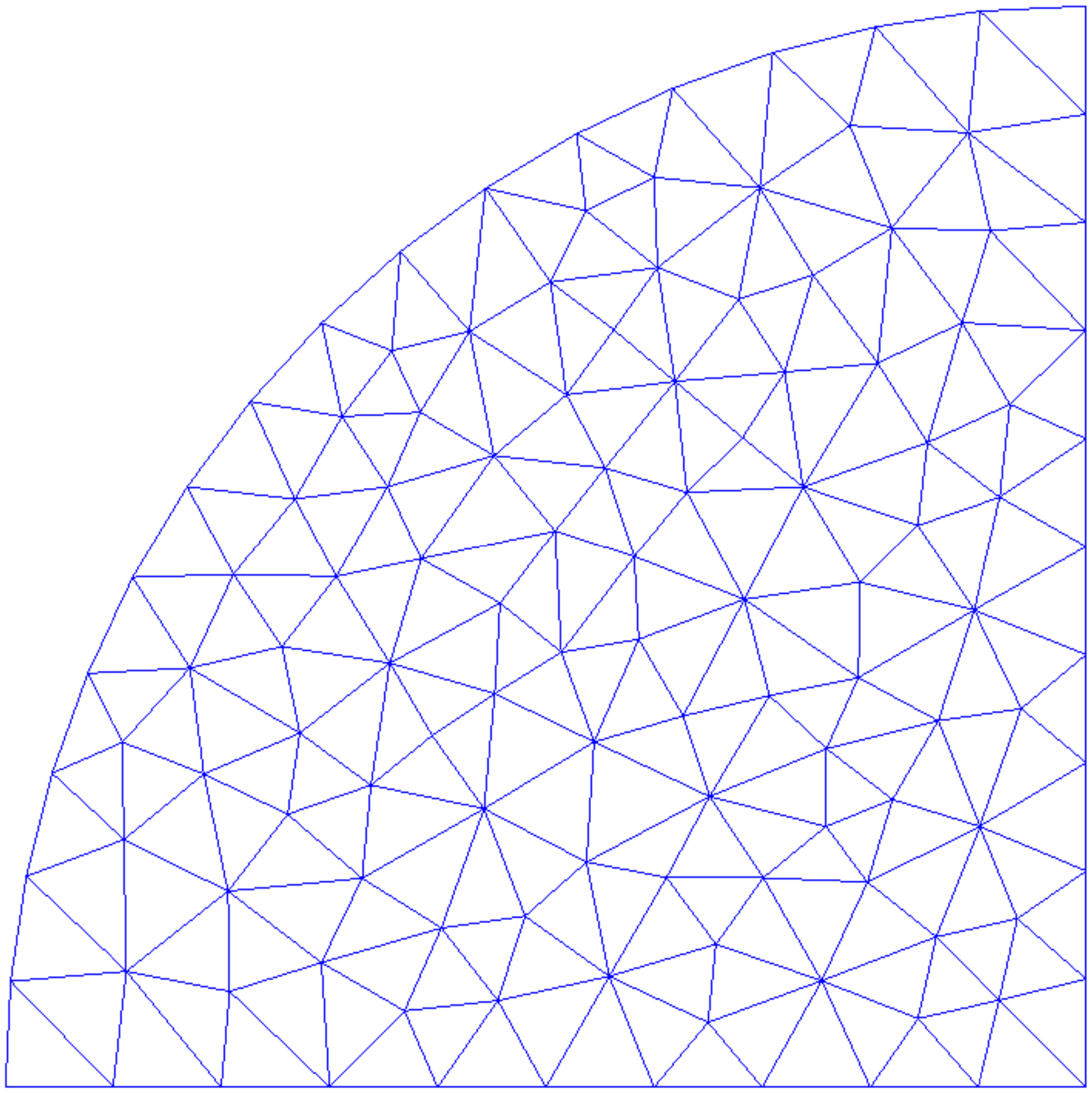}
  \end{center}
\end{minipage}\hfill
\begin{minipage}{0.33\linewidth}
  \begin{center}
    \includegraphics[width=.9\linewidth] {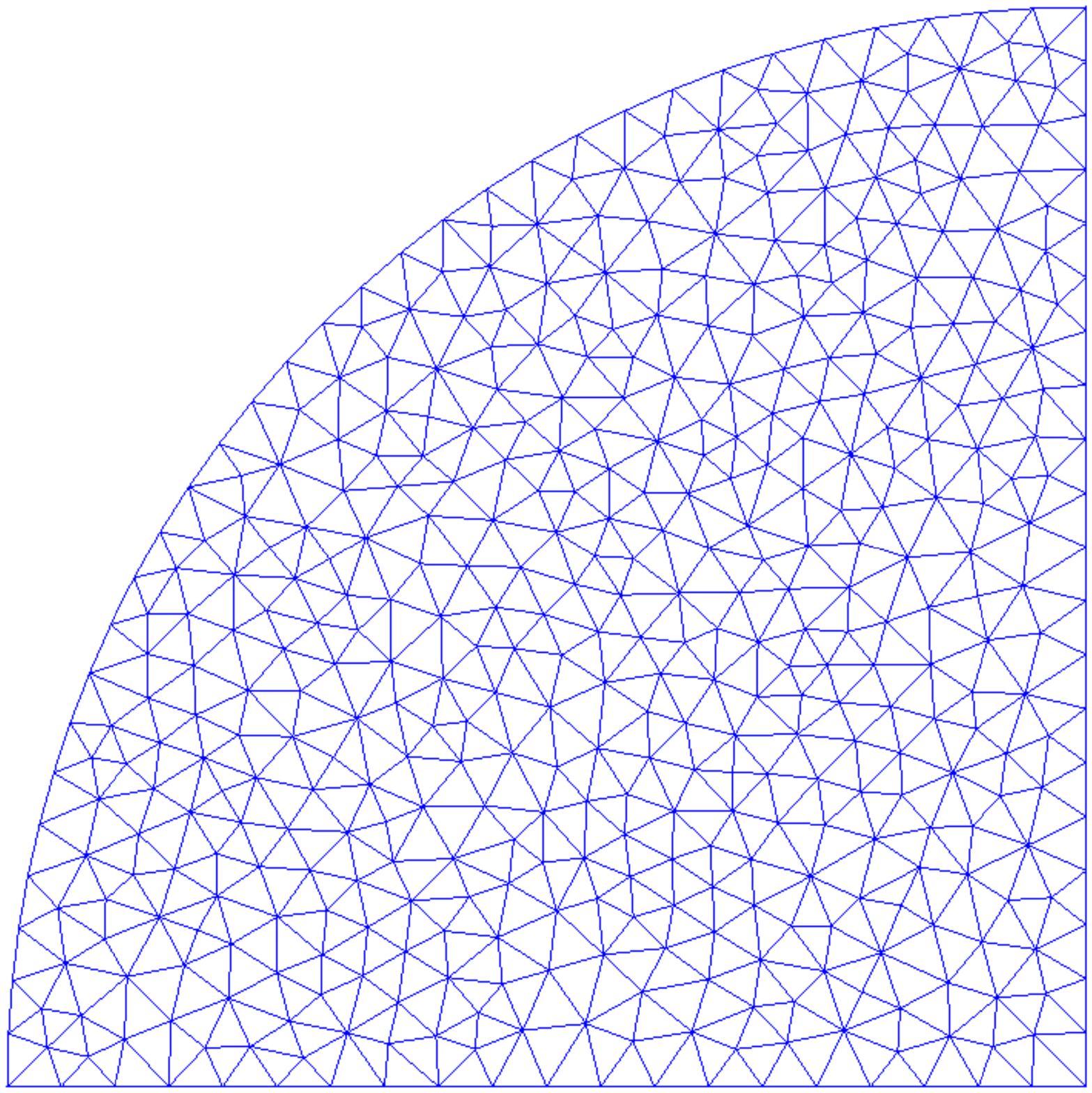}
  \end{center}
\end{minipage}\hfill
\begin{minipage}{0.33\linewidth}
  \begin{center}
    \includegraphics[width=.9\linewidth] {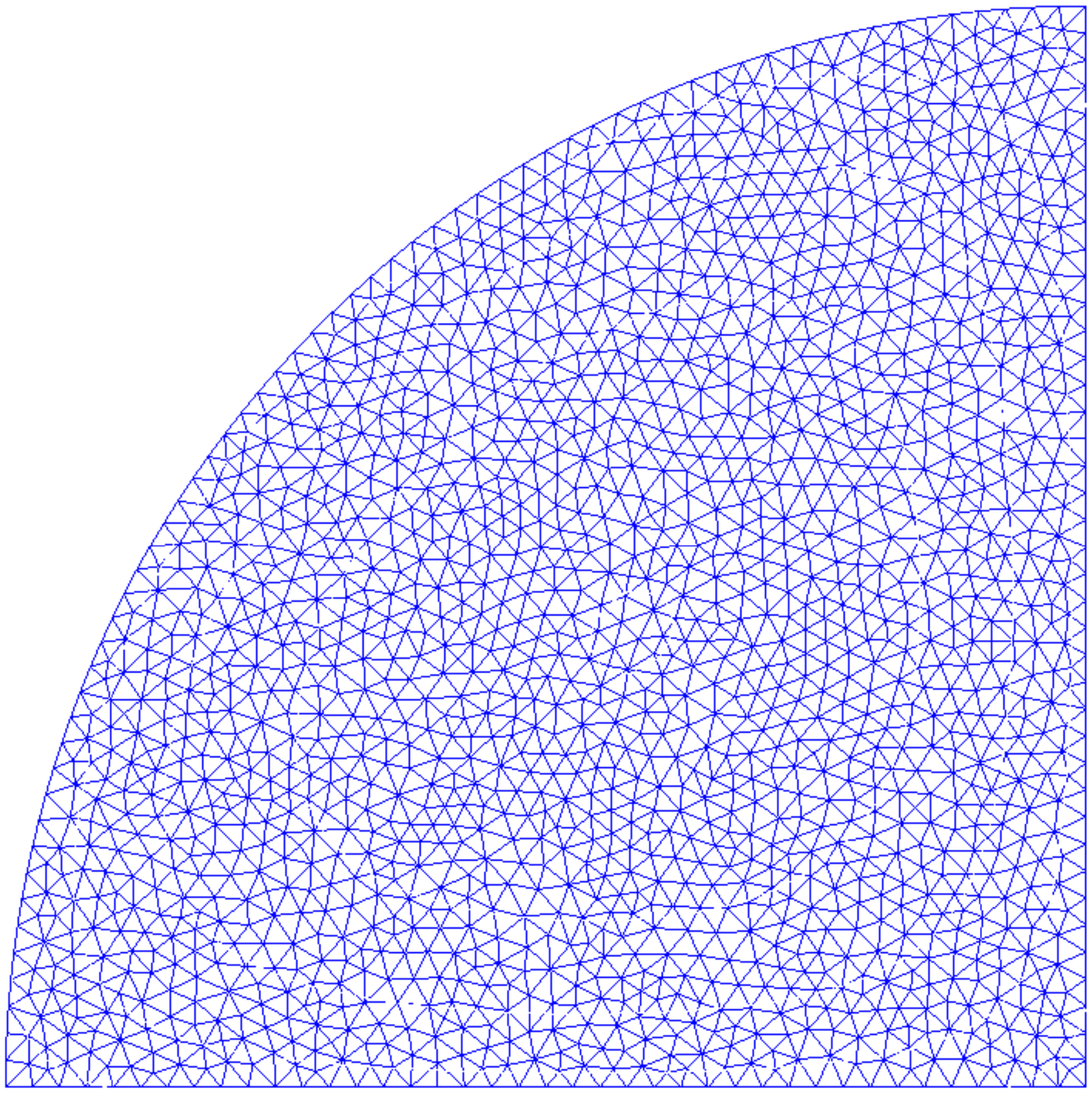}
  \end{center}
\end{minipage}\hfill
\\[1ex]
\begin{minipage}{0.33\linewidth}
  \begin{center}
    \includegraphics[width=.9\linewidth] {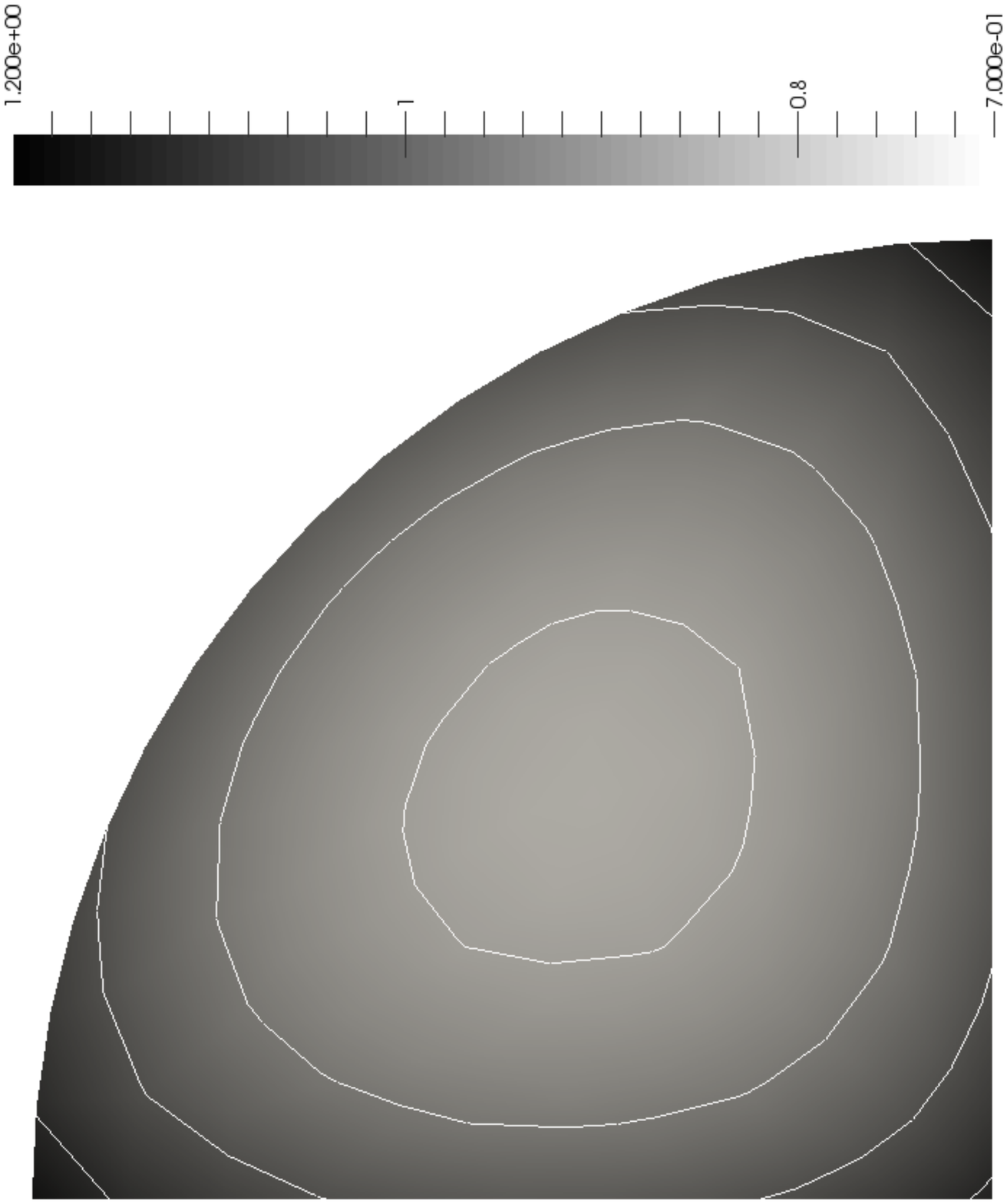}
  \end{center}
\end{minipage}\hfill
\begin{minipage}{0.33\linewidth}
  \begin{center}
    \includegraphics[width=.9\linewidth] {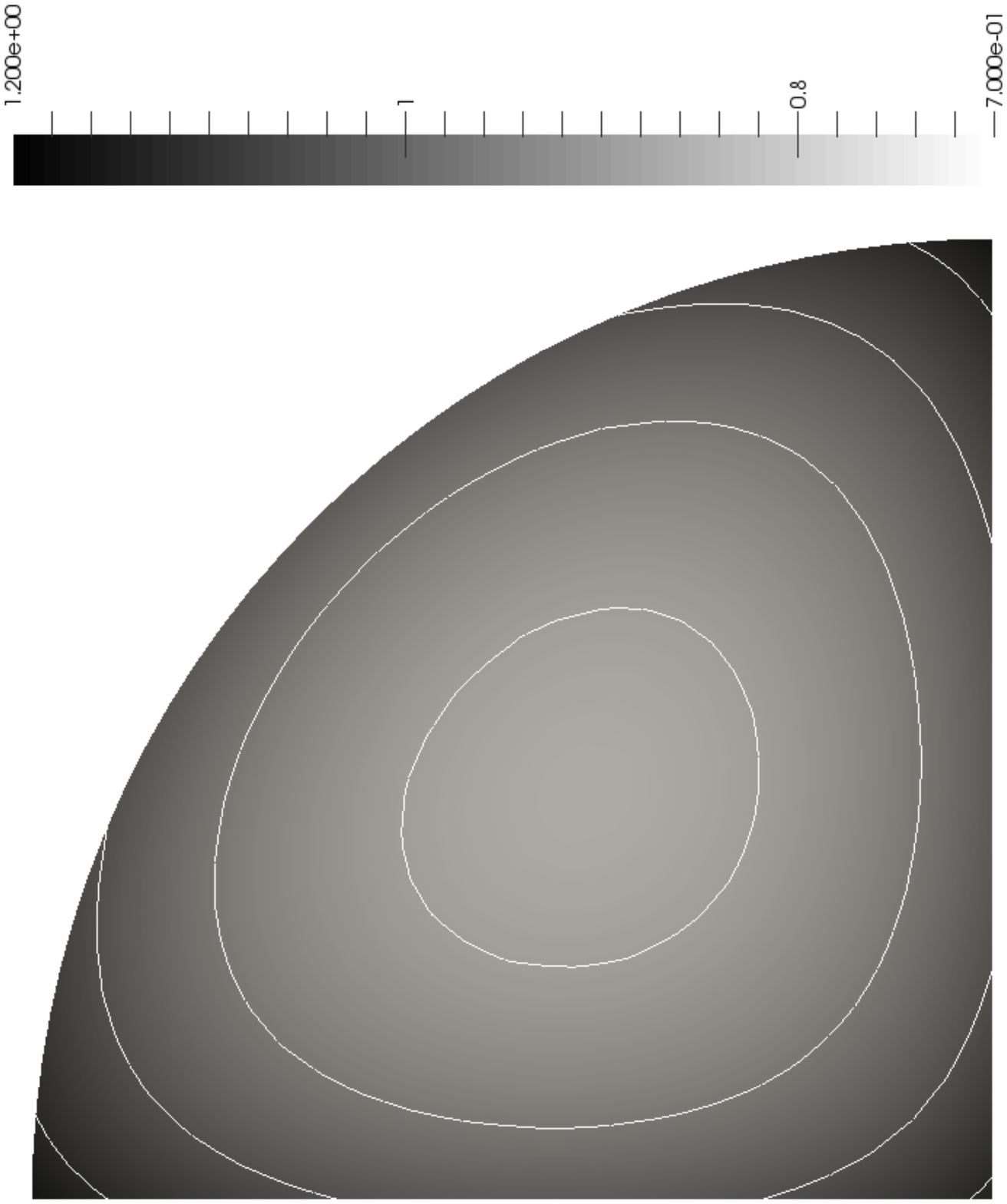}
  \end{center}
\end{minipage}\hfill
\begin{minipage}{0.33\linewidth}
  \begin{center}
    \includegraphics[width=.9\linewidth] {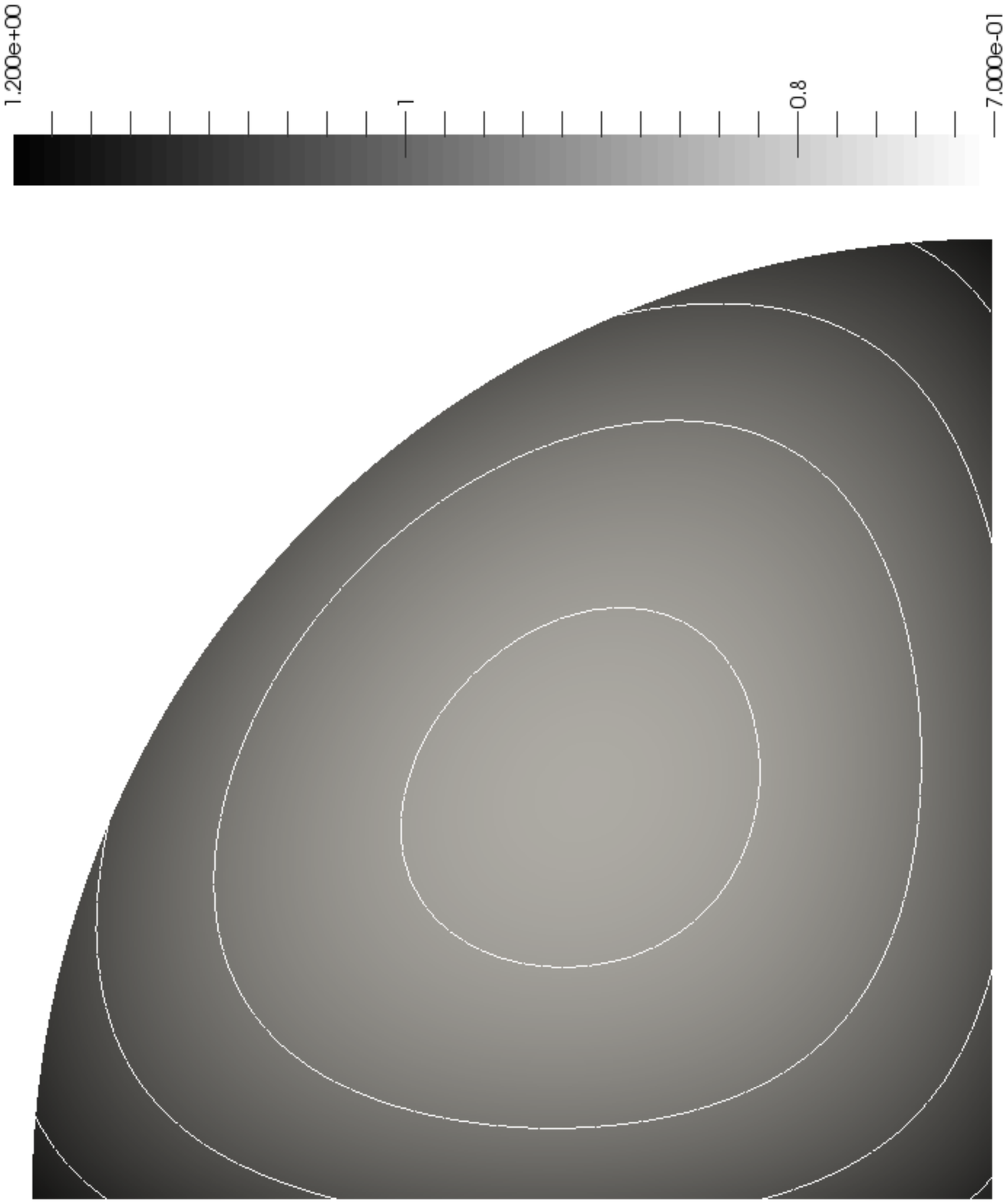}
  \end{center}
\end{minipage}\hfill
	\caption{
	Approximate solutions for $\alpha = 0.5$ on three meshes: 
	Left: grid 
	with 123 vertexes and 208 triangles; $\min w( x) = 0.7669, \ \max w( x) = 1.157$;
	Middle: grid 
	with 461 vertexes and 848 triangles; $\min w( x) = 0.7668, \ \max w( x) = 1.151$;
	Right: grid with 1731 vertexes and 3317 triangles;	$\min w( x) = 0.7668, \ \max w( x) = 1.151$.}
	\label{fig.combined}
\end{figure}

On Fig.~\ref{fig.1} we show the numerical solution level curves $u(x)=0.9, 0.8, 0.7, \dots$ 
for two limiting cases of $\alpha$, namely, the standard elliptic problem 
with Dirichlet data, $\alpha=0$, and with Neumann data, $\alpha = 1$.
These two examples are given for comparison with the cases of fractional boundary conditions.

Next, on Fig.~\ref{fig.combined} 
we show the computed solution obtained on three different grids for $\alpha = 0.5$ and $c_0 = 5$. 
One can observe a rather weak dependence
of the results on the grid size. Note, that in this case the solution  is smooth.
The solutions of the problem with $\alpha = 0.25$ and $\alpha = 0.75$ are shown on Fig.~\ref{fig.alpha}.
The impact of the coefficient $c_0$ can be observed  on Fig.~\ref{fig.c0}, where we show 
the approximate solution for  $c_0 = 1$ and $c_0 = 25$ that is
obtained on a grid with 461 vertexes,  a medium size grid.

\begin{figure}[h!]
\begin{minipage}{0.38\linewidth}
  \begin{center}
    \includegraphics[width=.9\linewidth] {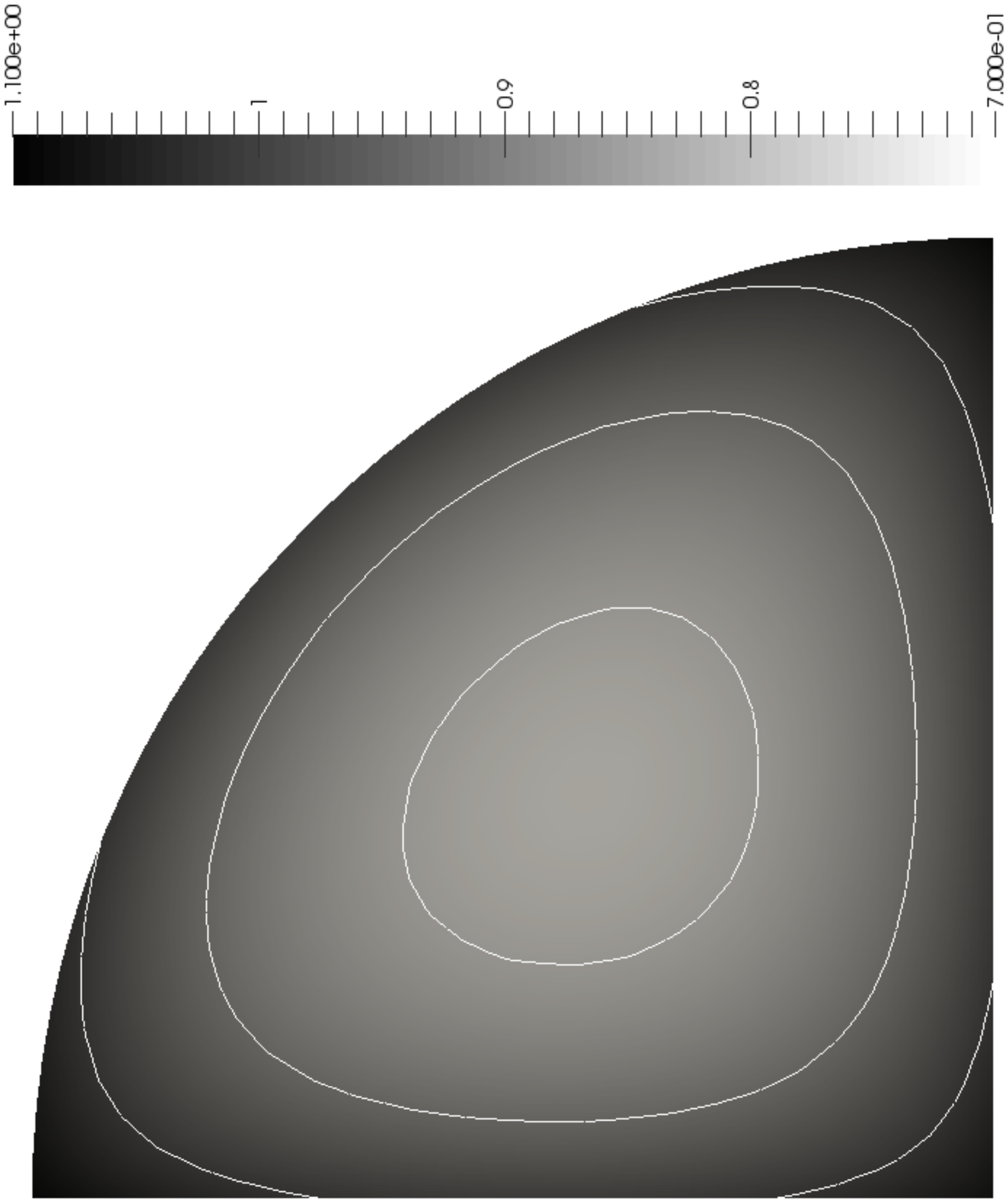}
  \end{center}
\end{minipage}\hfill
\begin{minipage}{0.38\linewidth}
  \begin{center}
    \includegraphics[width=.9\linewidth] {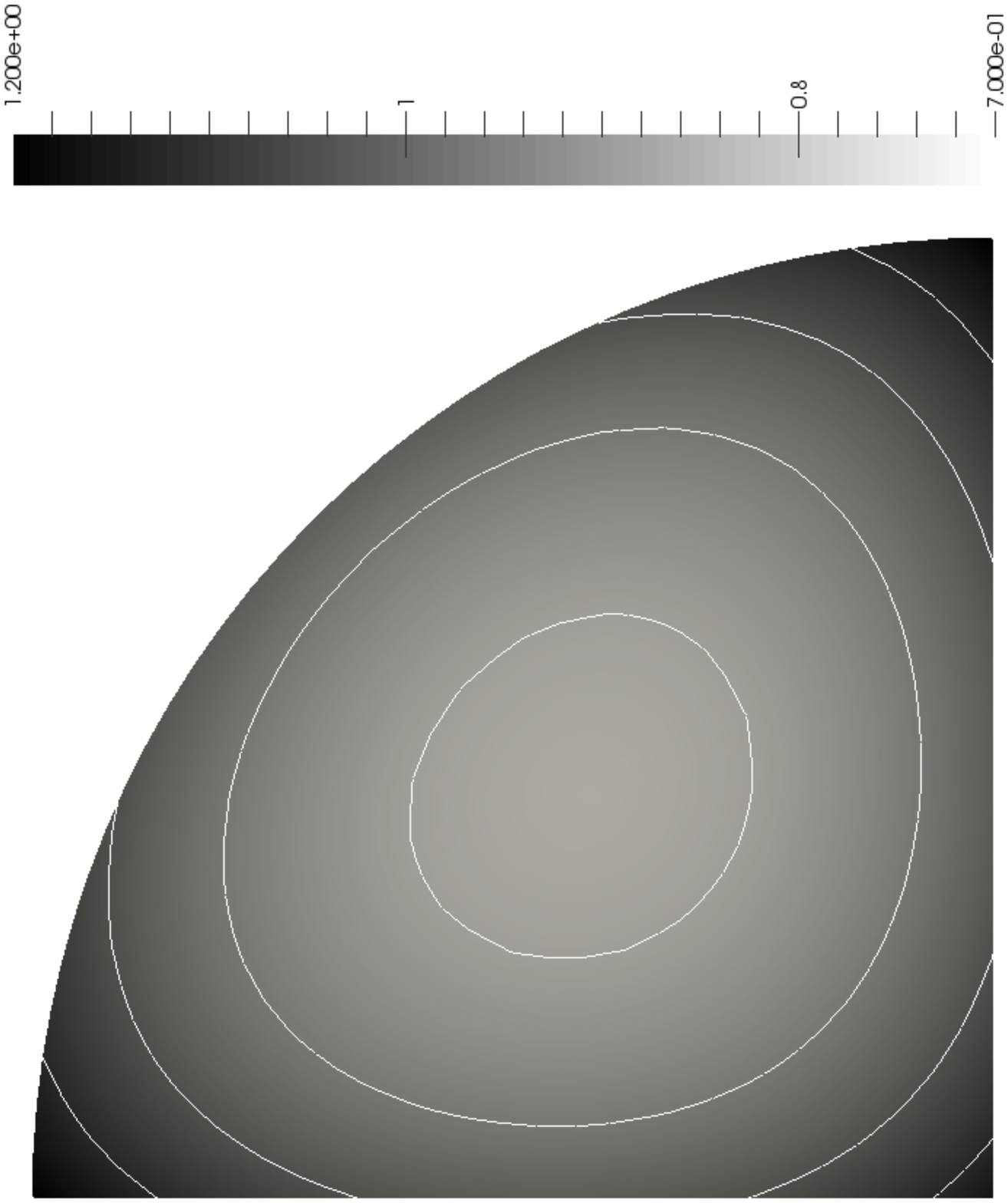}
  \end{center}
\end{minipage}
	\caption{The solution for different $\alpha$: Left:  $\alpha  = 0.25$, $\min w( x) = 0.767$, $  \max w( x) = 1.087$ 
         and  Right: $\alpha  = 0.75$, $\min w(x) = 0.769$, $\max w( x) =  1.201$.}
	\label{fig.alpha}
\end{figure}

\begin{figure}[h!]
\begin{minipage}{0.38\linewidth}
  \begin{center}
    \includegraphics[width=.9\linewidth] {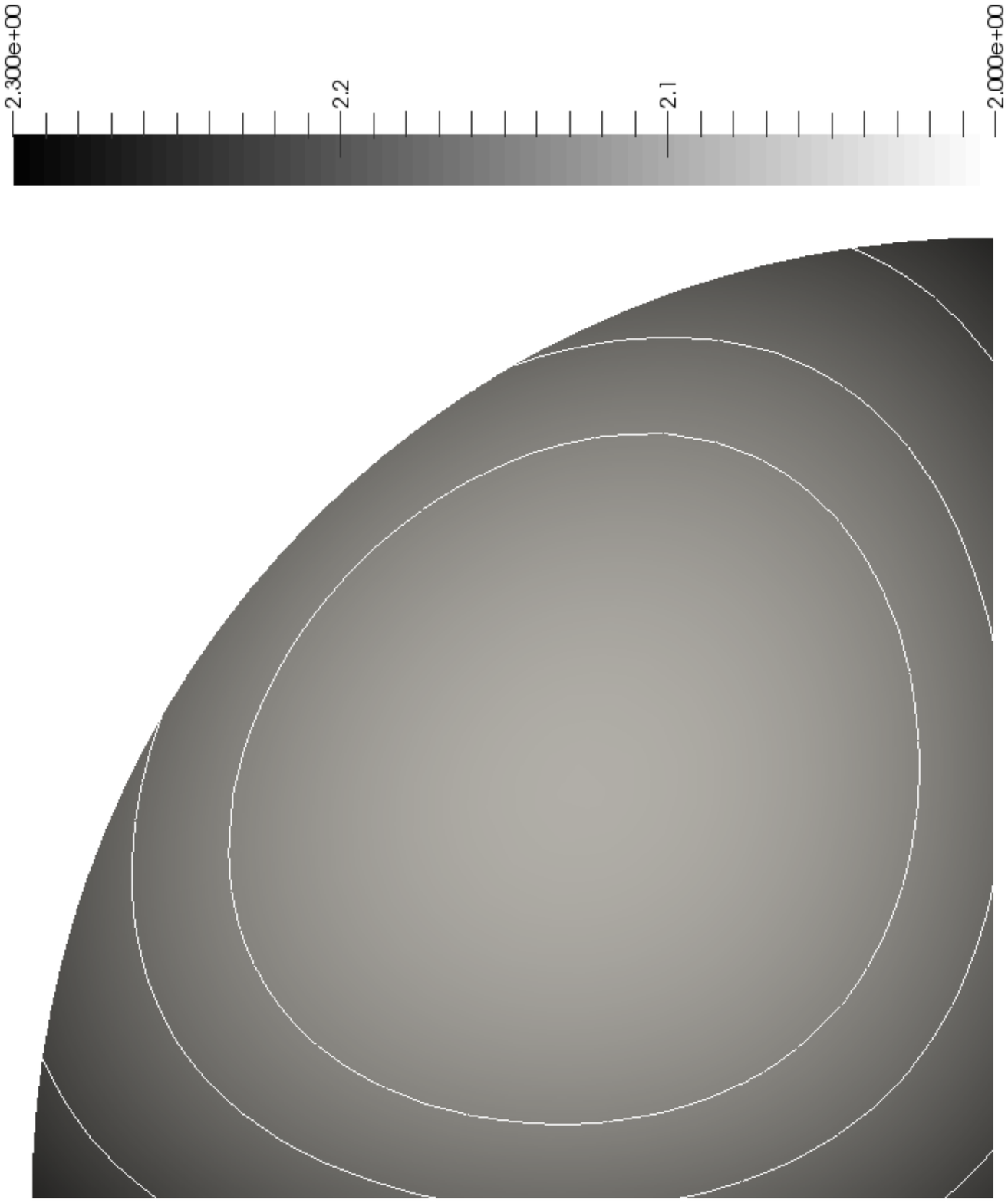}
  \end{center}
\end{minipage}\hfill
\begin{minipage}{0.38\linewidth}
  \begin{center}
    \includegraphics[width=.9\linewidth] {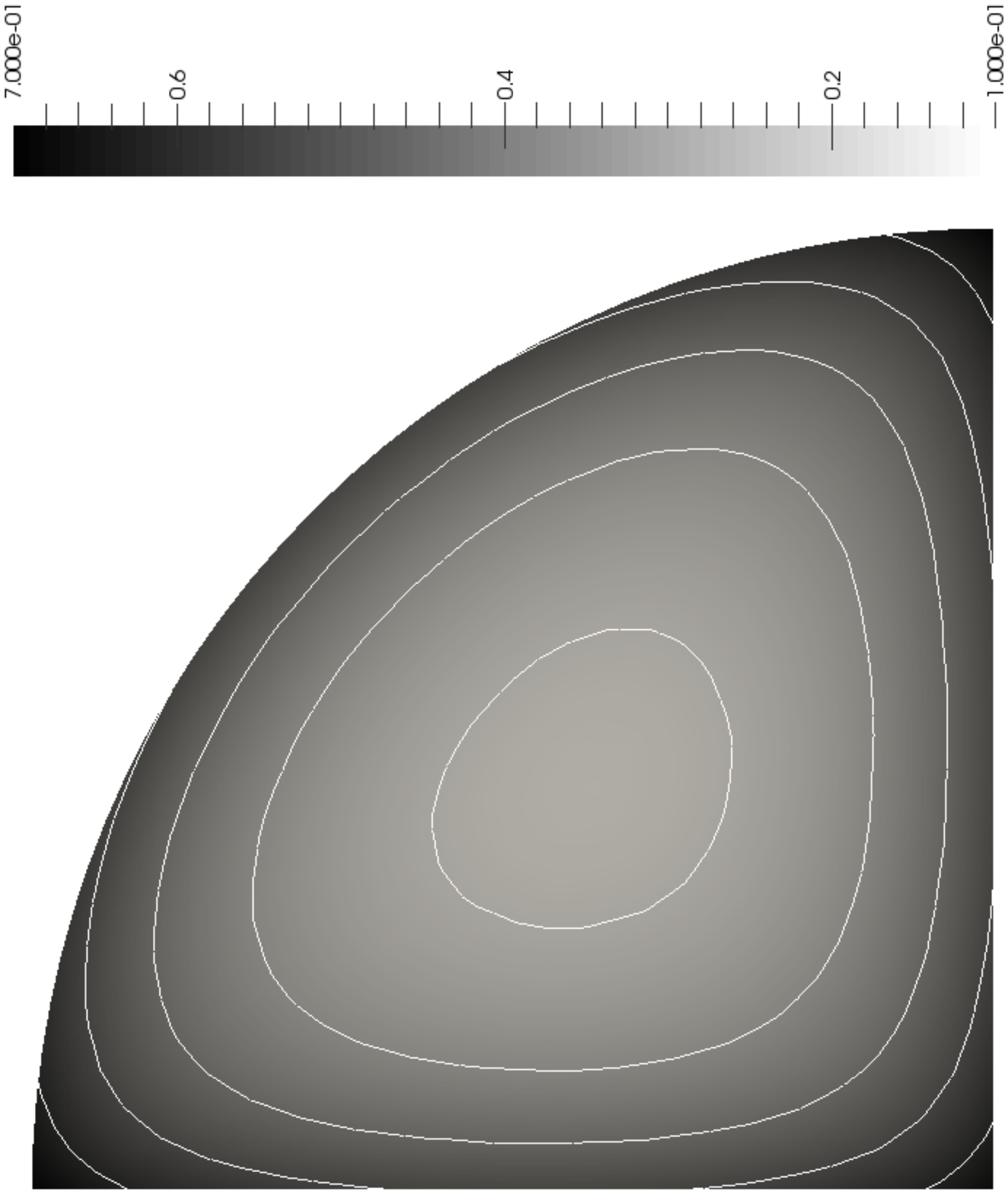}
  \end{center}
\end{minipage}
	\caption{The solution for different $c_0$: Left: $c_0 = 1$, $\min w( x) = 2.034$, $\max w( x) = 2.246$;
         Right: $c_0 = 25$ (right), $\min w( x) = 0.174$, $ \max w( x) =  0.690$.}
	\label{fig.c0}
\end{figure}

The most important issue concerning the efficiency and the accuracy of a numerical method for 
solving a boundary value problems for fractional power of 
elliptic operators or elliptic equations with fractional boundary conditions
is the impact of $M$, the number of nodes in the quadrature formula \eqref{12} or the
number of time steps in \eqref{23}.
Below we report the relative error  in $L^\infty(\Omega)$ and $L^2 (\Omega)$-norms:
\[
 e_\infty  = \max_{ x \in \Omega} |w_M(x) - w(x) |/\max_{x \in \Omega}|w(x)|,
 \quad   e_2 = \|w_M( x) -w( x)\|/\|w\|.
\] 
Here $w_M( x)$ is the numerical solution, and $w( x)$ is the reference (practically exact)
solution obtained for large $M$, 
shown in Fig.~\ref{fig.combined}.

The errors, $e_\infty $ and $e_2$, of the numerical solution of the problem for $c_0 = 5$,
$\alpha  = 0.25, 0.5, 0.75$ and various $M$
for both methods are shown in Table~\ref{tab-1:combined}. 
In these experiments for Method II we 
take $\delta=0.9$ and $\widetilde{\lambda}_1 \approx 0.949314$. 

First we note that 
to achieve acceptable accuracy with Method I, we need to use quadrature formulas with a fairly large number of nodes
($M \sim 100$). The highest convergence rates are observed for $\alpha  = 0.5$. The convergence decreases drastically
 when $\alpha $ gets close to $  0$ or $  1$. However, this numerical scheme has asymptotic exponential convergence,
 see, \cite{bonito2015numerical}. Analyzing  the numerical data  
 shown on Table \ref{tab-1:combined} we see that: (1) by doubling the quadrature points from $10$ to $20$
 the error reduces by a factor of $2$; (2) by doubling the quadrature points from $80$ to $160$
 the error reduces by  factors of $4 $ to $ 40$ for different $\alpha$.
 Also, for $\alpha=0.5$ by doubling the quadrature points from $40$ to $80$
 the error $e_2$ is reduced by a factor of $13$. The same conclusions can be made after analyzing the 
 numerical results of Table \ref{tab-2:combined} as well. This means that an exponential asymptotic 
 convergence rate begins to show for $M$ large enough depending on fractional power $\alpha$.
 
 The numerical results for Method II show, as predicted by the theory,
 almost second order convergence rate with respect to $\tau$.
 Note that relatively good accuracy is achieved  
even for small number of time steps $N$ ($N=M$). 
From these numerical experiments (performed on smooth solutions) we see that the 
Method II shows  better accuracy for relatively large time-step $\tau$, which
translates into fewer computations. 
 
The errors, $e_\infty $ and $e_2$,  of the numerical solution for problems with $\alpha  = 0.5$  and different 
values of the  coefficient $c_0$  are given in Table~\ref{tab-2:combined}. 
For Method II we 
take $\delta=0.2$ and $\widetilde{\lambda}_1 \approx 0.212867$ at $c_0=1$ and
$\delta=3$ and $\widetilde{\lambda}_1 \approx 3.170554$ at $c_0=25$.
From the numerical experiments we see that 
these methods are fairly insensitive to the variation of $c_0$.

\small
\begin{table}[!h]
\caption{Numerical solution error for $c_0=5$ and $\alpha=\frac14,\frac12,\frac34$}
\begin{center}
\begin{tabular}{|c|c ||ccc||ccc|} \cline{3-8} 
\multicolumn{2}{c ||}{~} & \multicolumn{3}{| c ||}{Method I} & \multicolumn{3}{|c|}{Method II} \\ \hline
$M$  & $e$ $\backslash$  $\alpha $  & 0.25 & 0.5 & 0.75  & 0.25 & 0.5 & 0.75   \\ \hline
5    & $e_\infty$       &  2.5009e-01  &  1.0492e-01 &  2.7720e-01  &  1.4875e-03  &  1.2809e-03 &  6.3801e-04 \\
     & $e_2$  	        &  2.6929e-01  &  1.0819e-01 &  2.7201e-01  &  2.1016e-04  &  2.3152e-04 &  1.4033e-04 \\   \hline
10   & $e_\infty$       &  1.5967e-01  &  4.4602e-02 &  1.7838e-01  &  5.3950e-04  &  4.2429e-04 &  1.9596e-04 \\
     & $e_2$  		&  1.7271e-01  &  4.6029e-02 &  1.7458e-01  &  5.6560e-05  &  6.1306e-05 &  3.6738e-05 \\   \hline
20   & $e_\infty$       &  8.4043e-02  &  1.2636e-02 &  9.4258e-02  &  1.8701e-04  &  1.3399e-04 &  5.7360e-05 \\
     & $e_2$  		&  9.1097e-02  &  1.3043e-02 &  9.2119e-02  &  1.4848e-05  &  1.5628e-05 &  9.2899e-06 \\   \hline
40   & $e_\infty$       &  3.3725e-02  &  2.0501e-03 &  3.7872e-02  &  6.1051e-05  &  4.0353e-05 &  1.6150e-05 \\
     & $e_2$  		&  3.6579e-02  &  2.1162e-03 &  3.6994e-02  &  4.0022e-06  &  3.9635e-06 &  2.3232e-06 \\   \hline
80   & $e_\infty$       &  9.2023e-03  &  1.5283e-04 &  1.0336e-02  &  1.8221e-05  &  1.1398e-05 &  4.3515e-06 \\
     & $e_2$  		&  9.9821e-03  &  1.5781e-04 &  1.0096e-02  &  1.0897e-06  &  1.0015e-06 &  5.7474e-07 \\   \hline
160  & $e_\infty$       &  1.4553e-03  &  3.7728e-06 &  1.6347e-03  &  4.9096e-06  &  2.9878e-06 &  1.1157e-06 \\
     & $e_2$  		&  1.5787e-03  &  3.9520e-06 &  1.5967e-03  &  2.8073e-07  &  2.4428e-07 &  1.3636e-07 \\   \hline
\end{tabular}
\end{center}
\label{tab-1:combined}
\end{table}
\normalsize
 
\small
\begin{table}[!h]
\caption{Numerical solution error for $\alpha  = 0.5$  and $c_0=1,5,25$}
\begin{center}
\begin{tabular}{|c|c ||ccc||ccc|} \cline{3-8} 
\multicolumn{2}{c ||}{~} & \multicolumn{3}{| c ||}{Method I} & \multicolumn{3}{|c|}{Method II} \\ \hline
$M$  & $e$ $\backslash$ $c_0 $ & 1        & 5            & 25            & 1             & 5            & 25   \\ \hline
5    & $e_\infty$       &  1.4067e-01  &  1.0492e-01 &  1.1232e-01  & 1.2847e-03  &  1.2809e-03   & 9.4560e-04 \\
     & $e_2$  	        &  1.4056e-01  &  1.0819e-01 &  1.3005e-01  & 3.1276e-04  &  2.3152e-04   & 1.2364e-04 \\   \hline
10   & $e_\infty$       &  6.0161e-02  &  4.4602e-02 &  4.7706e-02  & 5.2113e-04  &  4.2429e-04   & 3.0310e-04 \\
     & $e_2$  		&  6.0112e-02  &  4.6029e-02 &  5.5530e-02  & 1.0988e-04  &  6.1306e-05   & 3.3082e-05 \\   \hline
20   & $e_\infty$       &  1.7066e-02  &  1.2636e-02 &  1.3512e-02  & 1.8893e-04  &  1.3399e-04   & 8.9985e-05 \\
     & $e_2$  		&  1.7052e-02  &  1.3043e-02 &  1.5747e-02  & 3.2926e-05  &  1.5628e-05   & 8.8981e-06 \\   \hline
40   & $e_\infty$       &  2.7691e-03  &  2.0501e-03 &  2.1923e-03  & 6.2310e-05  &  4.0353e-05   & 2.4683e-05 \\
     & $e_2$  		&  2.7669e-03  &  2.1162e-03 &  2.5551e-03  & 8.8335e-06  &  3.9635e-06   & 2.3307e-06 \\   \hline
80   & $e_\infty$       &  2.0637e-04  &  1.5283e-04 &  1.6345e-04  & 1.9396e-05  &  1.1398e-05   & 6.3911e-06 \\
     & $e_2$  		&  2.0634e-04  &  1.5781e-04 &  1.9053e-04  & 2.2445e-06  &  1.0015e-06   & 5.7694e-07 \\   \hline
160  & $e_\infty$       &  5.1206e-06  &  3.7728e-06 &  4.0639e-06  & 5.6936e-06  &  2.9878e-06   & 1.5981e-06 \\
     & $e_2$  		&  5.1713e-06  &  3.9520e-06 &  4.7688e-06  & 5.5463e-07  &  2.4428e-07   & 1.2614e-07 \\   \hline
\end{tabular}
\end{center}
\label{tab-2:combined}
\end{table}

\normalsize
\smallskip
\section*{Acknowledgements}
The authors thank their institution for the support while working on this project.
The work of R. Lazarov was  partially supported also by grant NSF-DMS \# 1620318  
while the work of P. Vabishchevich was supported by the Ministry of 
Education and Science of the Russian Federation 
(Agreement \# 02.a03.21.0008).




\end{document}